\documentclass[12pt, a4paper]{article}
\usepackage{amsfonts,amsmath,amssymb,amsthm}
\usepackage{color, enumerate, graphicx, subfigure, authblk, tikz}
\numberwithin{equation}{section}

\newtheorem{lemma}{Lemma}[section]
\newtheorem{proposition}{Proposition}[section]

\newtheorem{theorem}{Theorem}[section]

\title{Global Dynamics and Existence of Traveling Wave Solutions for A Three-Species  Models}
\author[1]{Fanfan Li%\footnote{Research was partially supported by National Council of Science, Republic of China.}
}
\author[1]{Zhenlai Han%\footnote{Research was partially supported by National Council of Science, Republic of China.}
}
\author[2]{Ting-Hui Yang\footnote{Research was partially supported by Ministry of Science and Technology, Taiwan (R O C).}}
\affil[1]{School of Mathematical Sciences, University of Jinan, Jinan, Shandong 250022, P R China.}
\affil[2]{Department of Mathematics, Tamkang University, Tamsui Dist., New Taipei City 25137, Taiwan.}

\begin{document}
\date{}
\maketitle

\begin{abstract}
In this work, we investigate the system of three species ecological model involving one predator-prey subsystem coupling with a generalist predator with negative effect on the prey. Without diffusive terms, all global dynamics of  its corresponding reaction equations are proved analytically for all classified parameters. With diffusive terms, the transitions of different spatial homogeneous solutions, the traveling wave solutions, are showed by higher dimensional shooting method, the  Wazewski method. Some interesting numerical simulations are performed, and biological implications are given.
\end{abstract}

 
\section{Introduction}
%\keyword{}
%\textbf{Keywords}: $h$-deference equations, oscillation, fractional.
%\\
%\textbf{Mathematics Subject Classification 2010}: 34A10, 34C10, 26A33.}
In this work, we consider an ecological system of three species with diffusion as follows,
\begin{equation}\label{3pdemodel}
\left\{
\begin{aligned}
\partial_{t} u &=r_{1}u(1-u)-a_{12}uv-a_{13}uw,\\
\partial_{t} v &=r_{2}v(1-v)+a_{21}uv,\\
\partial_{t} w &=d\Delta w-\mu w+a_{31}uw,
\end{aligned}
\right.
\end{equation}
where parameters $d$ is the diffusive coefficient for species $w$, $r_{1}$ and $r_{2}$ are the intrinsic growth rates of species $u$ and $v$ respectively, and $\mu$ is the death rate of the predator $w$. The nonlinear interactions between species is the Lotka-Volterra type interactions between species where $a_{ij}(i<j)$ is the rate of consumption and $a_{ij}(i>j)$ measures the contribution of the victim (resource or prey) to the growth of the consumer \cite{Namba2008}. Here, the species $u$ is the nutrient resource of the predator-prey system, the species $v$ is called the generalist predator which can take advantage of various resources from two trophic levels, and the species $w$ is called the specialist predator which has a limited diet from $u$. To simplify the analysis, we only consider the species $w$ has diffusion effect.

System (\ref{3pdemodel}) without diffusive terms is given by the following system of three ODEs:
\begin{equation}\label{3odemodel}
\left\{
\begin{aligned}
\dot u&=r_{1}u(1-u)-a_{12}uv-a_{13}uw,\\
\dot v &=r_{2}v(1-v)+a_{21}uv,\\
\dot w &=-\mu w+a_{31}uw.
\end{aligned}
\right.
\end{equation}
The whole system can be seen as a predator-prey subsystem, $u$-$w$ subsystem, coupled with an extra species $v$ with negative effect on species $u$. The model can be used to describe parasitism, consumption or predation in the community of plants species \cite{Ford2010}. It is well-known that the ecological principle of competitive exclusion holds for the following classical two predators-one prey model \cite{Mcgehee1977},
\begin{equation*}
\left\{
\begin{aligned}
\dot u &=r_{1}u(1-u)-a_{12}uv-a_{13}uw,\\
\dot v &=-\mu_{1} v+a_{21}uv,\\
\dot w &=-\mu_{2} w+a_{31}uw.
\end{aligned}
\right.
\end{equation*}
However, by comparing these two models, system \eqref{3odemodel} is a modified two-predators one-prey model where one, $v$, is a generalist predator and another one, $w$, is a specialist predator. There are a fundamental difference, in Section 2, where we show the positive equilibrium of system (\ref{3odemodel}) can not only exist but it is also globally asymptotically stable, that is, two predators can co-exist.

Reaction-diffusion systems are often characterized by the existence of spatial homogeneous equilibria when the diffusion terms vanish. If there are more than one equilibrium, then we can expect a possible transition between them. These transitions are described by reaction-diffusion waves. Propagations of flames, migration of biological species, or tumor growth are among many examples of such phenomena \cite{Murray2003,Volpert1994}. In the PDE perspective, the existence of traveling wave solutions for reaction-diffusion systems in an important and interesting subject which has attracted considerable attentions \cite{Dunbar1983,Dunbar1984,Dunbar1986, Gardner1983, Jones1993, Hsu2012, Huang2012,Huang2016, Zou1997existence}. The phenomena of traveling wave solutions of reaction-diffusion systems have been widely studied \cite{Volpert1994} from the single equation with nonlinearity in monostable type \cite{Fisher1937} or bistable type \cite{KPP1937} to monotone systems \cite{liang2007asymptotic}. There have been great successes in the existence, uniqueness, stability and spreading speed of traveling wave solutions of monotone system \cite{Chen1997existence, liang2007asymptotic, Smith2000global}. 
\medskip

Unfortunately, our system which has an important nonlinear interaction, predator-prey type, between different species is  non-monotone. In the past three decades, by using different methods including the shooting method, Conley index and upper-lower solutions method, the existence of traveling wave solutions has been established for various predator-prey systems. See Dunbar \cite{Dunbar1983,Dunbar1984,Dunbar1986}, Gardner and Smoller \cite{Gardner1983}, Jones et al. \cite{Jones1993}, Huang \cite{Huang2012,Huang2016}, Hsu et al. \cite{Hsu2012}, Lin et al. \cite{Lin2015}, and references cited therein. In this work, we will use the so-called higher dimensional shooting method, Wazewski method to show the existence of positive traveling wave solutions from one unstable equilibrium to a stable one. Here we briefly describe this framework of the shooting method.\medskip

To show the existence of traveling wave solution, by using the moving coordinates, the reaction-diffusion system is transformed into a ODE system, and the existence of traveling wave solutions connecting two different equilibria is equivalent to a heteroclinic orbit of the corresponding ODE system. we analyze the structure of unstable manifold of the unstable equilibrium first. Then we construct a variant of Wazewski set $\Sigma$ with the unstable equilibrium as its boundary point and also containing the stable equilibrium. Third, dynamics of the system on all boundary of $\Sigma$ should be clarified. Next, pick up a curve contained in the unstable manifold with two end points on the ``exit set'' of boundary of $\Sigma$. It is clear that all solutions with initial conditions on this curve will attend to the unstable equilibrium as $t\to-\infty$. Then show that there exists a particular point on this curve, and the solution starting from this point will stay in the interior of $\Sigma$ for all $t\ge 0$.
%First, we discuss the stability of all boundary equilibrium and the global asymptotically stable of the positive equilibria of system (\ref{3odemodel}). Secondly, we find a heteroclinic orbit of the system (\ref{3odemodel}) connecting equilibria $E_{1}$ to $E_{*}$.
Finally, we define a nonempty subset of $\Sigma$ which contains the point stayed in $\Sigma$ for all positive time under the action of the ODE system. Then we can get our main result by constructing a Lyapunov function and use the LaSalle's invariance principle on this non-empty set. \medskip

In this work, our main contributions are as follows. First, for the corresponding reaction system (\ref{3odemodel}) we clarify completely the existence, non-existence, and all asymptotically states and their global stabilities are investigated theoretically. Secondly, we show the existence of traveling wave solutions  are obtained for a particular three species ecosystem with predator-prey interaction. We use the Wazewski principle to show the existence of  traveling wave solutions. Thought, the method is similar to \cite{Huang2012, Huang2016}, our system is three dimensional. Third,  numerical simulations are performed for some interesting initial functions. Finally, some biological interpretations are given.\medskip
%In this work, our main contributions are as follows:
%\begin{enumerate}
%\item ODE : All asymptotically states and their global stabilities are investigated theoretically.
%\item PDE : The existence of TWS solutions are obtained for a particular three species ecosystem with predator-prey interaction. We use the Wazewski principle to show the existence of TWS. Thought, the method is similar to \cite{Huang2012, Huang2016}, our system is three dimensional.
%\item Numerical simulations are performed for some interesting initial functions.
%\item Some biological interpretations are given.
%\end{enumerate}

The rest of this article is organized in the following manner. In Section 2, we first consider the corresponding reaction equations of (\ref{3pdemodel}) which is a system of three ODEs. The existence of boundary equilibria and coexistence equilibrium are obtained with some conditions. Moreover, we find the necessary and sufficient condition of global asymptotic stability of the positive equilibrium. In Section 3, by using shooting method, the existence and nonexistence  of traveling wave solutions of (\ref{3pdemodel}) are obtained. In Section 4, the numerical simulations are performed and presented. Finally, some remarks and discusses in biological meanings are also given in the last section.
%%%%%%%%%%%%%%%%%%%%%%%%%%%%%%%%%%
\section{Global Dynamics of the Corresponding ODE System}

In this section, we investigate dynamics of the ODE system \eqref{3odemodel} and the essential assumptions to guarantee the existence and local stabilities of all equilibria. Moreover, two extinction results and the global stability of positive equilibrium are showed.

\subsection{Preliminaries}
It is easy to see that $uv$-, $uw$-, and $vw$-planes are invariant subspaces of \eqref{3odemodel}. Hence solutions of \eqref{3odemodel} will be positive/non-negative if they start from a positive/non-negative point. Moreover, we can show that solutions are bounded.
\begin{lemma}\label{L-1}
The solutions of system (\ref{3odemodel}) are bounded.
\end{lemma}
\begin{proof}
From the first equation in system (\ref{3odemodel}) we have
$$
\frac{du}{dt}=r_{1}u(1-u)-a_{12}uv-a_{13}uw \leqslant r_{1}u(1-u),
$$
so that the comparison principle implies that
$$
\limsup_{t\rightarrow\infty}u\leqslant1,
$$
i.e. $0<u\leqslant1$.
Similarly we can get $0<v\leqslant1+\frac{a_{21}}{r_{2}}$.\\
Let $M=r_{1}+\mu$ and $D=\mu$. From the first and third equation in (\ref{3odemodel}) we have
\begin{equation*}
\begin{aligned}
\frac{d}{dt}\left(u+\frac{a_{13}}{a_{31}}w\right)&=r_{1}u(1-u)-a_{12}uv-a_{13}uw-\frac{a_{13}}{a_{31}}w+a_{13}uw\\
&\leq r_{1}u-\frac{a_{13}}{a_{31}}w\\
&=Mu-D(u+\frac{a_{13}}{a_{31}}w)\\
&\leq M-D(u+\frac{a_{13}}{a_{31}}w).
\end{aligned}
\end{equation*}
Using the comparison principle again, we have
$$
\limsup_{t\rightarrow\infty}\left(u+\frac{a_{13}}{a_{31}}w\right)\leq \frac{M}{D}.
$$
\end{proof}

\subsection{Assumptions and Two Extinction Results}

From now on, we always make the assumptions,
\begin{enumerate}[{\rm\ (H1)}]
\item  $r_{1}>a_{12}$,
\item  $a_{31}>\mu$,
\end{enumerate}
which will be used in the rest of the article, because of the following two extinction results.

\begin{lemma}\label{L-2}
If $r_{1}\leq a_{12}$, then $\lim_{t\rightarrow\infty}u(t)=0$ and $\lim_{t\rightarrow\infty}w(t)=0$.
\end{lemma}
\begin{proof}
It is easy to see that if $\lim_{t\rightarrow\infty}u(t)=0$ then $\lim_{t\rightarrow\infty}w(t)=0$, sequently. Hence we only show that the first limit holds. Two cases, $r_{1}< a_{12}$ and $r_{1}=a_{12}$, are considered.
It is easy to see that $v(t)\ge 1$ eventually by comparison principle,
since
\[ \dot v =r_{2}v(1-v)+a_{21}uv\ge r_{2}v(1-v). \]

For $r_{1}< a_{12}$ and $t$ large enough, we have
$$
\frac{\dot{u}}{u}=r_{1}(1-u)-a_{12}v-a_{13}w<r_{1}-a_{12}<0.
$$
Hence we have $u(t)<ce^{(r_{1}-a_{12})t}\to 0$ as $t\rightarrow\infty$.

For $r_{1}=a_{12}$, we obtain that $u(t)$ is decreasing with respect to $t$, and claim that $\lim_{t\to\infty}u(t)=0.$
Suppose to the contrary that $\lim_{t\to\infty}u(t)=\xi>0$. By Markus limiting theorem \cite{Markus1956}, we have
\begin{equation*}
\begin{aligned}
0=\liminf_{t\to\infty}\dot{u}(t)&=\liminf_{t\to\infty}\big(r_{1} u(t)(1-u(t))-r_1u(t)v(t)-a_{13}u(t)w(t)\big)\\
&=r_{1}\xi(1-\xi) -r_1\xi\limsup_{t\to\infty}v(t)-a_{13}\xi\limsup_{t\to\infty}w(t)\\
&=\xi\big(r_{1}(1-\xi) -r_1\limsup_{t\to\infty}v(t)-a_{13}\limsup_{t\to\infty}w(t)\big)=0,
\end{aligned}
\end{equation*}
which implies
\begin{align}\label{ine}
r_1(1-\xi) =r_1\limsup_{t\to\infty}v(t)+a_{13}\limsup_{t\to\infty}w(t).
\end{align}
However, by the second equation of \eqref{3odemodel} and the comparison principle, it is clear that $\liminf_{t\to\infty}v(t)\ge1$. Hence, by \eqref{ine}, we obtain that
\[ a_{13}\limsup_{t\to\infty}w(t)=r_1(1-\xi)- r_1\limsup_{t\to\infty}v(t)\le r_1(1-\xi)- r_1<0\]
which contradicts to the positivity of $w(t)$. Hence we complete the proof.
\end{proof}

\begin{lemma}\label{L-3}
If $a_{31}\leq \mu$, then $\lim_{t\rightarrow\infty}w(t)=0$.
\end{lemma}
\begin{proof}
Similarly, two cases, $a_{31}< \mu$ and $a_{31}=\mu$, are considered. For case $a_{31}<\mu$, we have
$$
\frac{\dot{w}}{w}=a_{31}u-\mu\le a_{31}-\mu<0
$$
holds. Hence $w(t)\to 0$ as $t\rightarrow\infty$.

For case $a_{31}=\mu$,
\[ \dot{w}=w(a_{31}u-\mu)\le w(a_{31}-\mu)=0. \]
Hence $w$ is decreasing, and we claim that $\lim_{t\to\infty}w(t)=0$.
Suppose to the contrary that $\lim_{t\rightarrow\infty}w(t)=\xi>0$. By comparison principle again, we have
\begin{equation*}
\begin{aligned}
0=\liminf_{t\rightarrow\infty}\dot{w}(t)&=\liminf_{t\rightarrow\infty}\big(-\mu w(t)+a_{31}u(t)w(t)\big)\\
&=a_{31}\xi \liminf_{t\rightarrow\infty}u(t)-\mu\xi\\
&=a_{31}\xi\big(\liminf_{t\rightarrow\infty}u(t)-1\big)=0,
\end{aligned}
\end{equation*}
which implies $\lim_{t\to\infty}u(t)=1$ and $\lim_{t\to\infty}\dot{u}(t)=0$. However, by considering the first equation of \eqref{3odemodel}, we obatin
\[ 0=\lim_{t\rightarrow\infty}\dot{u}(t)=\lim_{t\to\infty}\big(r_{1}u(1-u)-a_{12}uv-a_{13}uw\big)\le-a_{13}\xi<0, \]
where this is a contradiction. Hence we have $\lim_{t\rightarrow\infty}w(t)=0$, and the proof is complete.
\end{proof}

Biologically, these two results can be easily interpreted in the biological point of view. From the first equation of \eqref{3odemodel}, species $u$ have two negative effects from $v$ and $w$, respectively. To sustain the negative effect of species $v$, $r_1>a_{12}$, is necessary for survival of $u$ and supporting for $w$ in Lemma \ref{L-2}. Alternatively, for Lemma \ref{L-3}, if the mortality rate $\mu$ of species $w$ is greater than the benefit getting from species $u$, the conversion rate $a_{31}$, then $w$ will die out eventually. Whenever $r_1\le a_{12}$ or $\mu\le a_{31}$ hold, then  system (\ref{3odemodel}) is reduced to a one- or two-dimensional subsystem of \eqref{3odemodel} which is well studied by classical results like Poincare-Bendixson Theorem. Hence we make assumptions (H1) and (H2).

\subsection{Equilibria and Stability in $\mathbb{R}^{3}$}
By straightforward calculation, we obtain that there are one trivial equilibrium $E_{0}=(0,0,0)$, and four semi-trivial equilibria, $E_{1}=(1,0,0)$, $E_{2}=(0,1,0)$, 
\begin{align*}
E_{12}&=(u^*_{12},v^*_{12},0)=\left(\frac{r_{2}(r_{1}-a_{12})}{r_{1}r_{2}+a_{12}a_{21}},\frac{r_{1}r_{2}+r_{1}a_{21}}{r_{1}r_{2}+a_{12}a_{21}},0\right) \quad\text{ and  }\\
E_{13}&=(u^*_{13},0,w^*_{13})=\left(\frac{\mu}{a_{31}},0,\frac{r_{1}(a_{31}-\mu)}{a_{13}a_{31}}\right),
\end{align*}
of system (\ref{3odemodel}). Here $u^*_{12}$ and $v^*_{12}$ satisfy the equations,
\begin{equation}\label{u12}
\left\{
\begin{aligned}
r_{1}(1-u^*_{12})-a_{12}v^*_{12}&=0,\\
r_{2}(1-v^*_{12})+a_{21}u^*_{12}&=0, 
\end{aligned}
\right.
\end{equation}
and $u^*_{13}$ and $w^*_{13}$ have the forms,
\begin{equation}\label{u13}
\begin{aligned}
u^*_{13}=\mu/a_{13} \quad \text{ and} \quad r_{1}(1-u^*_{13})=a_{13}w^*_{13}.
\end{aligned}
\end{equation}

It is obvious that the equilibria, $E_{0}$, $E_{1}$ and $E_{2}$, always exist without any restriction. By contrast, the equilibria $E_{12}$ and $E_{13}$ exist if assumptions (H1) and (H2) hold, respectively. The positive equilibrium
\begin{align}\label{Estar}
E_{*}&=(u_{*},v_{*},w_{*})\nonumber\\
&=\left(\frac{\mu}{a_{31}},1+\frac{a_{21}\mu}{a_{31}r_{2}},\frac{r_{1}r_{2}a_{31}-r_{1}r_{2}\mu-a_{12}a_{31}r_{2}-a_{12}a_{21}\mu}{a_{13}a_{31}r_{2}}\right)
\end{align}
exists if
\begin{enumerate}[{\rm (H3)}]
\item  $r_{1}r_{2}a_{31}-r_{1}r_{2}\mu-a_{12}a_{31}r_{2}-a_{12}a_{21}\mu>0.$
\end{enumerate}
holds. It is hard to see clearly the biological meanings of assumption (H3) because of the complicated form. However, it is easy to see that the inequality $r_1\le a_{12}$ implies that (H3) does not hold, that is, (H3) is a sufficient conditions of (H1). Similarly, assumption (H3) is also a sufficient conditions of (H2). Biologically, there are two key points for the existence of positive equilibrium $E_*$. One is the survival of species $u$ by its $r$-strategy to overcome the negative effect from species $v$, and another one is species $w$ should overcome its mortality by getting benefit from species $u$.\medskip

By direct computations, we have the Jacobian matrix of system (\ref{3odemodel}) given by
\begin{equation}\label{jac}
J=
\begin{bmatrix}
r_{1}-2r_{1}u-a_{12}v-a_{13}w&-a_{12}u&-a_{13}u\\
a_{21}v&r_{2}-2r_{2}v+a_{21}u&0\\
a_{31}w&0&-\mu+a_{31}u
\end{bmatrix}.
\end{equation}
\begin{enumerate}[(i)]
\item It is clear that
\begin{equation*}
J(E_{0})=
\begin{bmatrix}
r_{1}&0&0\\
0&r_{2}&0\\
0&0&-\mu
\end{bmatrix}
\end{equation*}
has two positive eigenvalues and one negative eigenvalue, and $E_0$ is saddle.

\item Evaluating \eqref{jac} at $E_1$ implies that
\begin{equation*}
J(E_{1})=
\begin{bmatrix}
-r_{1}&-a_{12}&-a_{13}\\
0&r_{2}+a_{21}&0\\
0&0&-\mu+a_{31}
\end{bmatrix}
\end{equation*}
has two positive eigenvalues and one negative eigenvalue. Similarly, we can obtain the matrix
 \begin{equation*}
J(E_{2})=
\begin{bmatrix}
r_{1}-a_{12}&0&0\\
a_{21}&-r_{2}&0\\
0&0&-\mu
\end{bmatrix}
\end{equation*}
which is stable if assumption (H1) does not hold. Actually, we can show that, by Markus limiting theorem \cite{Markus1956}, equilibrium $E_2$ is globally asymptotically stable if (H1) does not hold.

\item The Jacobian evaluated at $E_{12}$ is given by
 \begin{equation}\label{JE12}
J(E_{12})=
\begin{bmatrix}
-r_{1}u^*_{12}&-a_{12}u^*_{12}&-a_{13}u^*_{12}\\
a_{21}v^*_{12}&-r_{2}v^*_{12}&0\\
0&0&-\mu+a_{31}u^*_{12}
\end{bmatrix}.
\end{equation}
It is easy to see that there are two eigenvalues, $\lambda_{1}$ and $\lambda_{2}$, corresponding to the upper-left $2\times 2$ submatrix of \eqref{JE12} and one eigenvalue, $\lambda_{3}=-\mu+a_{31}u^*_{12}$, with
\begin{align*}
\lambda_{1}+\lambda_{2}&=-(r_{1}u^*_{12}+r_{2}v^*_{12})<0,\\
\lambda_{1}\lambda_{2}&=(r_{1}r_{2}+a_{12}a_{21})u^*_{12}v^*_{12}>0.
\end{align*}
Thus the matrix has  two negative eigenvalues and one positive eigenvalue if $\mu<a_{31}u^*_{12}$, or three negative eigenvalues if $\mu>a_{31}u^*_{12}$.

\item The Jacobian evaluated at $E^*_{13}$ is given by
\begin{equation}\label{jac2}
J(E^*_{13})=
\begin{bmatrix}
-r_{1}u^*_{13}&-a_{12}u^*_{13}&-a_{13}u^*_{13}\\
0&r_{2}+a_{21}u^*_{13}&0\\
a_{31}w^*_{13}&0&0
\end{bmatrix}.
\end{equation}
Observing the form of the Jacobian matrix \eqref{jac2}, there are one positive eigenvalue, $\lambda_{3}=r_2+a_{21}u^*_{13}$, and  two eigenvalues, $\lambda_1$ and $\lambda_2$, which are obtained by removing the second column and the second row. Although it is obvious that
\begin{align*}
\lambda_{1}+\lambda_{2}&=-r_1u^*_13<0 \quad\text{ and}\\
\lambda_{1}\lambda_{2}&=a_13a_31u^*_13w^*_13>0,
\end{align*}
equilibrium $E_{13}$ is saddle.
\item The Jacobian evaluated at $E_{*}$ is given by
 \begin{equation*}
J(E_{*})=
\begin{bmatrix}
-r_{1}u_{*}&-a_{12}u_{*}&-a_{13}u_{*}\\
a_{21}v_{*}&-r_{2}v_{*}&0\\
a_{31}w_{*}&0&0
\end{bmatrix}.
\end{equation*}
Then the characteristic equation is
\begin{equation*}
\begin{aligned}
\lambda^3+(r_1u_*+r_2v_*)\lambda^2+&(r_1r_2u_*v_*+a_{12}a_{21}u_*v_*+a_{13}a_{31}u_*w_*)\lambda\\&+r_2a_{13}a_{31}u_*v_*w_*=0.
\end{aligned}
\end{equation*}
It is obviously that, by Routh-Hurwitz criterion, the real parts of three roots of the characteristic equation are all negative if and only if
\[ 
(r_1u_*+r_2v_*)(r_1r_2u_*v_*+a_{12}a_{21}u_*v_*+a_{13}a_{31}u_*w_*)>r_2a_{13}a_{31}u_*v_*w_*
\]
which is clearly true. Hence the positive equilibrium $E_*$ is stable whenever it exists.
\end{enumerate}
Let us summarize the above local stability of all equilibria in the following proposition.
\begin{proposition}
\begin{enumerate}[{\rm (i)}]
\item The trivial equilibrium $E_0$ is unstable.
\item The semi-trivial equilibrium $E_{1}$ is unstable.
\item The semi-trivial equilibrium $E_{2}$ is globally asymptotically stable if {\rm (H1)} does not hold.
\item The semi-trivial equilibrium $E_{12}$ is stable if $\mu>a_{31}u^*_{12}$.
\item The semi-trivial equilibrium $E_{13}$ is unstable.
\item The positive equilibrium $E_{*}$ exists and is stable if {\rm (H3)} holds.
\end{enumerate}
\end{proposition}
Furthermore, we can obtain the following two global results.
\begin{theorem}\label{E12GAS}
Let assumption {\rm (H1)} and $\mu\ge a_{31}u^*_{12}$ hold. Then the positive equilibria $E_{12}$ exists, and it is globally asymptotically stable.
\end{theorem}
\begin{proof}
Modify the standard Lyapunov function in the following form
\[
V(u(t),v(t),w(t))=\frac{a_{21}}{a_{12}}(u-u^*_{12}\ln u)+v-v^*_{12}\ln v+\frac{a_{13}a_{21}}{a_{12}a_{31}}w.
\]
Then, with equation \eqref{u12},
\begin{align*}
\frac{d}{dt}&V(u(t),v(t),w(t))\\
&=\frac{a_{21}}{a_{12}}(u-u^*_{12})(r_1(1-u)-a_{12}v-a_{13}w)+\\
&\quad (v-v^*_{12})(r_2(1-v)+a_{21}u)+\frac{a_{13}a_{21}}{a_{12}a_{31}}(-\mu w+a_{31}u)\\
&=-\frac{r_1a_{21}}{a_{12}}(u-u^*_{12})^2-r_2(v-v^*_{12})^2+\frac{a_{13}a_{21}}{a_{12}}(\frac{-\mu}{a_{31}}+u^*_{12})w\le 0.
\end{align*}
Hence $\{dV/dt=0\}=\{(u^*_{12}, v^*_{12}, 0)\}$ if $\mu> a_{31}u^*_{12}$, or $\{dV/dt=0\}=\{W\ge 0 : (u^*_{12}, v^*_{12}, W)\}$ if $\mu= a_{31}u^*_{12}$. However, the maximal invariant set of $\{dV/dt=0\}$ is the singleton set $\{(u^*_{12}, v^*_{12}, 0)\}$ for these two possibilities.
By LaSalle's Invariance Principle, we show the global stability of equilibrium $E_{12}$ if $\mu\ge a_{31}u^*_{12}$. The proof is completed.

\end{proof}
\begin{theorem}\label{EstarGAS}
Let assumption {\rm (H3)} hold. Then the positive equilibria $E_{*}$ exists, and it is globally asymptotically stable.
\end{theorem}
\begin{proof}
Define the Lyapunov function in the following form
\[
V(u(t),v(t),w(t))=\frac{a_{12}}{a_{21}}(u-u_{*}\ln u)+v-v_{*}\ln v+\frac{a_{13}a_{21}}{a_{12}a_{31}}(w-w_{*}\ln w).
\]
Then 
\begin{equation*}
\begin{aligned}
\frac{d}{dt}&V(u(t),v(t),w(t))\\
=&\frac{a_{12}}{a_{21}}(u-u_{*})(r_{1}(1-u)-a_{12}v-a_{13}w)+(v-v_{*})(r_{2}(1-v)+a_{21}u)\\
&+\frac{a_{13}a_{21}}{a_{12}a_{31}}(w-w_{*})(-\mu+a_{31}u)\\
=&\frac{a_{12}}{a_{21}}(u-u_{*})(r_{1}(u_{*}-u)+a_{12}(v_{*}-v)+a_{13}(w_{*}-w))\\
&+(v-v_{*})(r_{2}(v_{*}-v)+a_{21}(u-u_{*}))+\frac{a_{13}a_{21}}{a_{12}a_{31}}(w-w_{*})(a_{31}(u-u_{*}))\\
=&-\frac{r_{1}a_{21}}{a_{12}}(u-u_{*})^{2}-r_{2}(v-v_{*})^{2}\leq0.
\end{aligned}
\end{equation*}
Hence by the LaSalle's Invariance Principle, the $\omega$-limit set of any solution of (\ref{3odemodel}) is contained in the maximal invariant subset of $\{dV/dt=0\}=\{(u_{*}, v_{*},W): W>0\}$, which is the singleton $\{(u_{*}, v_{*},w_{*})\}$. We complete the proof.
\end{proof}
%%%%%%%%%%%%%%%%%%%%%%%%%%%%%%%%%%
\section{Existence of traveling wave solutions}

In this section, motivated by Huang \cite{Huang2012, Huang2016}, the high dimensional shooting method is implemented to investigate the existence of wave fronts, or traveling wave solutions of \eqref{3pdemodel} from $E_1$ to $E_*$. Following the ideas of Huang, we list the main steps as follows.
\begin{enumerate}[(i)]
\item By using the moving coordinates, the reaction-diffusion system is transformed into a ODE system.
% and the existence of traveling wave solutions connecting two different equilibria is equivalent to a heteroclinic orbit of the corresponding ODE system. 
\item Construct a variant of Wazewski set $\Sigma$ with $E_1$ as its boundary point and also containing $E_*$. Dynamics of the system on all boundaries of $\Sigma$ should be clarified. 

\item Analyze the structure of unstable manifold of $E_1$, and pick up a curve contained in the unstable manifold of $E_1$ with two end points on the ``exit set'' of boundary of $\Sigma$. 
%It is clear that all solutions with initial conditions on this curve will attend to the unstable equilibrium as $t\to-\infty$. 
\item Show that there exists a particular point on this curve, and the solution starting from this point will stay in the interior of $\Sigma$ for all $t\ge 0$.
%First, we discuss the stability of all boundary equilibrium and the global asymptotically stable of the positive equilibria of system (\ref{3odemodel}). Secondly, we find a heteroclinic orbit of the system (\ref{3odemodel}) connecting equilibria $E_{1}$ to $E_{*}$.
\item Define a nonempty subset of $\Sigma$ which contains the point stayed in $\Sigma$ for all positive time under the action of the ODE system. Then by constructing a Lyapunov function and using the LaSalle's invariance principle, we obtain the main result.
\end{enumerate}

\subsection{The ODE forms and the Lienard Transformation}
We consider the solution of \eqref{3pdemodel} with the moving coordinate $\xi$ and wave speed $c$ of the form
\begin{equation}\label{profile}
\left\{
\begin{aligned}
u(x, t) = U(x  + ct)=U(\xi),\\
v(x, t) = V (x  + ct)=V(\xi),\\
w(x, t) = W (x  + ct)=W(\xi),
\end{aligned}
\right.
\end{equation}
satisfying the asymptotical boundary conditions from $E_1$ to $E_*$, that is,
\begin{equation}\label{AB}
\begin{aligned}
\lim_{\xi\to -\infty}\big(U(\xi), V(\xi), W(\xi)\big)&=(1, 0, 0), \quad\text{ and }\\
\lim_{\xi\to \infty}\big(U(\xi), V(\xi), W(\xi)\big)&=(u_*, v_*, w_*).
\end{aligned}
\end{equation}

 A direct computation shows that $(U(\xi),V(\xi), W(\xi))$ is a traveling wave solution of \eqref{3pdemodel} if and only if $(U (\xi ), V (\xi ), W(\xi))$ is a solution of the system,
\begin{equation}\label{profile}
\left\{
\begin{aligned}
cU'&=U\left(r_{1}(1-U)-a_{12}V-a_{13}W\right),\\
cV'& =V(r_{2}(1-V)+a_{21}U),\\
cW'& =dW''+W(-\mu+a_{31}U).
\end{aligned}
\right.
\end{equation}
For a constant $c > 0$ we consider the Lienard transformation to make the following changes of variables and scaling:
\begin{equation*}
\begin{aligned}
X_{1}(t) & =U(ct), \\
X_{2}(t) & =V(ct),\\
Y(t) & =W(ct),\\
Z(t) & = \frac{1}{c}[cW(ct)-d W'(ct)].
\end{aligned}
\end{equation*}
Then, upon straightforward computation, \eqref{profile} is transformed to a four dimensional system,
\begin{equation}\label{profile-2}
\begin{aligned}
\dot X_{1} &= X_{1}\left(r_{1}(1-X_{1})-a_{12}X_{2}-a_{13}Y\right),\\
\dot X_{2} &= X_{2}(r_{2}(1-X_{2})+a_{21}X_{1}), \\
\dot Y &=  \rho[Y-Z],\ \ \ \ \rho=\frac{c^2}{d},\\
\dot Z &= Y(-\mu+a_{31}X_{1}).
\end{aligned}
\end{equation}
It is clear that $(U(\xi),V(\xi),W(\xi))$ is a nonnegative solution of (\ref{profile}) with asymptotical boundary conditions \eqref{AB} if and only if $\left(X_{1}(t),X_{2}(t),Y(t),Z(t)\right)$ is a nonnegative solution of (\ref{profile-2}) satisfying the asymptotical boundary conditions,
$$
\left(X_{1}(-\infty),X_{2}(-\infty),Y(-\infty),Z(-\infty)\right)=(1,0,0,0),
$$
$$
\left(X_{1}(\infty),X_{2}(\infty),Y(\infty),Z(\infty)\right)=(u_{*},v_{*},w_{*},w_{*}).
$$

\subsection{The Wazewski set and its Exist Subset}
Let $\sigma_1$ and $\sigma_2$ be constants defined by
$$
\sigma_{1}=\frac{\rho+\sqrt{\rho^{2}-4\rho (a_{31}-\mu)}}{2\rho}=\frac{c+\sqrt{c^2-4d(a_{31}-\mu)}}{2c}
$$
and
$$
\sigma_{2}=\frac{\rho+\sqrt{\rho^{2}+4\rho \mu}}{2\rho}.
$$
It is clear that $0<\sigma_{1}<1$ and $\sigma_{2}>1$ if $c> c_*\equiv 2\sqrt{d(a_{31}-\mu)}$ or, equivalently, $\rho> 4(a_{31}-\mu)$. For $c> c_*$, we define a wedged like region $\Sigma \subset \mathbb{R}^{4}$ as follows, 
$$
\Sigma=\{(X_{1}, X_{2}, Y, Z): 0\leq X_{1}\leq 1,0\leq X_{2}\leq 1+\frac{a_{21}}{r_{2}}, Y\geq 0, \sigma_{1}Y\leq Z\leq \sigma_{2} Y\}.
$$
Then the boundary of $\Sigma$ consists of surfaces $P_{1}$, $P_{2}$, $Q_{1}\sim Q_{5}$ represented by
\begin{equation*}%\label{profile-3}
\begin{aligned}
Q_{1}&=\{X_{1}=0, 0\leq X_{2}\leq 1+\frac{a_{21}}{r_{2}}, Y\geq 0, \sigma_{1}Y\leq Z\leq \sigma_{2} Y\},\\
Q_{2}&=\{X_{1}=1, 0\leq X_{2}\leq 1+\frac{a_{21}}{r_{2}}, Y\geq 0, \sigma_{1}Y\leq Z\leq \sigma_{2} Y\},\\
Q_{3}&=\{0\leq X_{1}\leq 1,X_{2}=0, Y\geq 0, \sigma_{1}Y\leq Z\leq \sigma_{2} Y\},\\
Q_{4}&=\{0\leq X_{1}\leq 1,X_{2}=1+\frac{a_{21}}{r_{2}}, Y\geq 0, \sigma_{1}Y\leq Z\leq \sigma_{2} Y\},\\
Q_{5}&=\{0< X_{1}< 1,0< X_{2}<1+\frac{a_{21}}{r_{2}}, Y=Z=0\},\\
P_{1}&=\{0<X_{1}<1, 0<X_{2}<1+\frac{a_{21}}{r_{2}}, Y>0,Z=\sigma_{1}Y\},\\
P_{2}&=\{0<X_{1}<1, 0<X_{2}<1+\frac{a_{21}}{r_{2}}, Y>0,Z=\sigma_{2}Y\}.\\
\end{aligned}
\end{equation*}
The vector field of (\ref{profile-2}) has a very simple property in the boundary of $\Sigma$, which can be characterized by the following two lemmas.
\begin{lemma}\label{L-4}
Let $c> c_*$ and $\Phi_{t}(p)$ be the flow of (\ref{profile-2}), i.e. $\Phi_{t}(p)$ is a solution of (\ref{profile-2}) satisfying the initial condition $\Phi_{0}(p)=p\in \mathbb{R}^{4}$. Then for any $p\in Int(\Sigma)$, $\Phi_{t}(p)$ cannot leave $\Sigma$ from a point in the boundary $\cup_{i=1}^{5}Q_{i}$ of $\Sigma$ at any positive time.
\end{lemma}
\begin{proof}
It is obvious that the boundary $Q_{1}$, $Q_{3}$ and $Q_{5}$ are invariant sets of (\ref{profile-2}). Hence any solution of (\ref{profile-2}) through a point in the interior of $\Sigma$ can not leave $\Sigma$ from a point in the set $Q_{1}\cup Q_{3}\cup Q_{5}$.

First suppose $p_{0}=(1, X_2(0), Y(0), Z(0))\in Q_{2}$, then the first equation of (\ref{profile-2}) yields that at $p_{0}$,
$$
\dot{X}_{1}(0)=-a_{12}X_{2}(0)-a_{13}Y(0)<0.
$$
So that the vector field of (\ref{profile-2}) points interior of $\Sigma$ and hence $\Phi_{t}(p_0)$ cannot exit $\Sigma$ at $p_{0}$ in the set $Q_{2}$.

Next suppose $p_{1}=(X_1(0), 1+\frac{a_{21}}{r_2}, Y(0), Z(0))\in Q_{4}$, then the second equation of (\ref{profile-2}) yields that at $p_{1}$,
\begin{equation*}
\begin{aligned}
\dot{X}_{2}(0)&=\left(1+\frac{a_{21}}{r_{2}}\right)\left[r_{2}\left(1-\left(1+\frac{a_{21}}{r_{2}}\right)\right)+a_{21}X_{1}(0)\right]\\
&=-a_{21}-\frac{a_{21}^{2}}{r_{2}}+\left(a_{21}+\frac{a_{21}^{2}}{r_{2}}\right)X_{1}(0)\\
&<-a_{21}-\frac{a_{21}^{2}}{r_{2}}+a_{21}+\frac{a_{21}^{2}}{r_{2}}=0.
\end{aligned}
\end{equation*}
So that the vector field of (\ref{profile-2}) points interior of $\Sigma$ and hence $\Phi_{t}(p_1)$ cannot
exit $\Sigma$ at $p_{1}$ in the set $Q_{4}$.
\end{proof}
Next let us study the vector field at the sets $P_{1}$ and $P_{2}$.
\begin{lemma}\label{L-5}
Let $c> c_*$, where $c$ is the wave speed in the system (\ref{profile-2}). Then the vector field of (\ref{profile-2}) at $p \in P_{1}\cup P_{2}$ points outside of $\Sigma$.
\end{lemma}
\begin{proof}
First consider a point $p_{1}=(X_1(0), X_2(0), Y(0), Z(0))\in P_{1}$. Then $Z(0)=\sigma_{1}Y(0)<Y(0)$ and $Y(0)>0$. By the last two equations of (\ref{profile-2}), we obtain
$$
\dot{Y}(0)=\rho[Y(0)-Z(0)]>0,
$$
\begin{equation*}
\begin{aligned}
\frac{\dot{Z}(0)}{\dot{Y}(0)}&=\frac{Y(0)(-\mu+a_{31}X_{1}(0))}{\rho[Y(0)-Z(0)]}\\
&=\frac{-\mu+a_{31}X_{1}(0)}{\rho(1-\sigma_{1})}\\
&<\frac{a_{31}-\mu}{\rho(1-\sigma_{1})}=\sigma_{1}.\\
\end{aligned}
\end{equation*}
It implies that the vector field at the point $p_{1}\in P_{1}$ points outside of $\Sigma$. Next, let $p_{2}=(X_1(0), X_2(0), Y(0), Z(0))\in P_{2}$. Then $Z(0)=\sigma_{2}Y(0)>Y(0)$ and $Y(0)>0$. At $p_{2}$, we have
$$
\dot{Y}(0)=\rho[Y(0)-Z(0)]<0,
$$
\begin{equation*}
\begin{aligned}
\frac{\dot{Z}(0)}{\dot{Y}(0)}&=\frac{Y(0)(-\mu+a_{31}X_{1}(0))}{\rho[Y(0)-Z(0)]}\\
&=\frac{-a_{31}X_{1}(0)+\mu}{\rho(\sigma_{2}-1)}\\
&<\frac{\mu}{\rho(\sigma_{2}-1)}=\sigma_{2}.
\end{aligned}
\end{equation*}
It implies that the vector field at the point $p_{2}\in P_{2}$ points outside of $\Sigma$.
\end{proof}
By Lemma \ref{L-4} and Lemma \ref{L-5}, we can see that with initial point $p\in\Sigma$ $\Phi_{t}(p)$ can only leave $\Sigma$ from a point in $P_{1}\cup P_{2}$ at some positive time. 

\subsection{The Unstable Manifold of $E_1$}
Now we turn to study the unstable manifold of the equilibrium $E_{1}$. The linearized system of (\ref{profile-2}) at $E_{1}$ is
\begin{equation}\label{linearized}
\begin{aligned}
\dot{X_{1}}&=-r_{1}X_{1}-a_{12}X_{2}-a_{13}Y,\\
\dot{X_{2}}&=(r_{2}+a_{21})X_{2},\\
\dot{Y}&=\rho[Y-Z]\\
\dot{Z}&=(-\mu+a_{31}) Y,
\end{aligned}
\end{equation}
with Jacobian matrix of (\ref{profile-2}) at $E_{1}$,
 \begin{equation*}
\begin{bmatrix}
-r_{1}&-a_{12}&-a_{13}&0\\
0&r_{2}+a_{21}&0&0\\
0&0&\rho&-\rho\\
0&0&-\mu+a_{31}&0
\end{bmatrix}.
\end{equation*}
For $c> c_*$, upon a direct computation, the matrix has one negative eigenvalue
$$
\lambda_{0}=-r_{1}
$$
and three distinct positive eigenvalues
\begin{equation}\label{eigenvalue}
\begin{aligned}
\lambda_{1}&=r_{2}+a_{21},\\
\lambda_{2}&=\frac{\rho+\sqrt{\rho^{2}-4\rho(a_{31}-\mu)}}{2},\\
\lambda_{3}&=\frac{\rho-\sqrt{\rho^{2}-4\rho(a_{31}-\mu)}}{2},
\end{aligned}
\end{equation}
where the corresponding eigenvectors to $\lambda_{1}$, $\lambda_{2}$ and $\lambda_{3}$ are
\begin{equation}\label{eigenvectors}
\begin{aligned}
h_{1}&=\left[-\frac{a_{12}}{\lambda_{1}+r_{1}},1,0,0\right]^{T},\\
h_{2}&=\left[-\frac{a_{13}}{\lambda_{2}+r_{1}},0,1,\frac{a_{31}-\mu}{\lambda_{2}}\right]^{T},\\
h_{3}&=\left[-\frac{a_{13}}{\lambda_{3}+r_{1}},0,1,\frac{a_{31}-\mu}{\lambda_{3}}\right]^{T}.
\end{aligned}
\end{equation}
Hence, the unstable manifold $\varepsilon_{1}^{u}$ of $E_{1}$ is tangent to the plane
$$
P=\{k_{1}h_{1}+k_{2}h_{2}+k_{3}h_{3}+E_{1}:k_{1},k_{2},k_{3}\in\mathbb{R}\}=\{x_{1},x_{2},y,z\}
$$
where
\begin{equation*}
\begin{aligned}
x_{1}&=1-\frac{k_{1}a_{12}}{\lambda_{1}+r_{1}}-\frac{k_{2}a_{13}}{\lambda_{2}+r_{1}}-\frac{k_{3}a_{13}}{\lambda_{3}+r_{1}},\\
x_{2}&=k_{1},\\
y&=k_{2}+k_{3},\\
z&=\frac{k_{2}(a_{31}-\mu)}{\lambda_{2}}+\frac{k_{3}(a_{31}-\mu)}{\lambda_{3}}.
\end{aligned}
\end{equation*}
\begin{lemma}\label{L-6}
There are two points $p_{1}^{*}\in P_{1}$ and $p_{2}^{*}\in P_{2}$ and a curve $\gamma\in \varepsilon_{1}^{u}$ which connects $p_{1}^{*}$ and $p_{2}^{*}$ such that
$$
\gamma\setminus\{p_{1}^{*}, p_{2}^{*}\}\subset {\rm Int}(\Sigma).
$$
\end{lemma}
\begin{proof}
Note that the transformations: $(x_{2},y,z)\mapsto(k_{1},k_{2},k_{3})$,
\begin{equation*}
\begin{aligned}
x_{2}&=k_{1},\\
y&=k_{2}+k_{3},\\
z&=\frac{k_{2}(a_{31}-\mu)}{\lambda_{2}}+\frac{k_{3}(a_{31}-\mu)}{\lambda_{3}}
\end{aligned}
\end{equation*}
is invertible. Thus the plane can also be expressed as
$$
P=\left\{\left(1-\frac{x_{2}a_{12}}{\lambda+r_{2}}-\frac{\lambda_{2}\lambda_{3}z+r_{1}(a_{31}-\mu)y}{(a_{31}-\mu)(\lambda_{2}+r_{1})(\lambda_{3}+r_{1})},x_{2},y,z\right):x_{2},y,z\in\mathbb{R}\right\}.
$$F
Since the unstable manifold $\varepsilon_{1}^{u}$ is tangent to $P$ at $E_{1}$, there is a small $k>0$ such that the unstable manifold $\varepsilon_{1}^{u}$ of $E_{1}$ can be expressed as
$$
\varepsilon_{1}^{u}=\left\{\left(x_{1}(x_{2},y,z),x_{2},y,z\right):(x_{2},y,z)\in[0,k]^{3}\right\}
$$
where
\begin{equation}\label{x1}
\begin{aligned}
x_{1}(x_{2},y,z)=1-\frac{x_{2}a_{12}}{\lambda+r_{2}}-\frac{\lambda_{2}\lambda_{3}z+r_{1}(a_{31}-\mu)y}{(a_{31}-\mu)(\lambda_{2}+r_{1})(\lambda_{3}+r_{1})}+\eta(x_{2},y,z),
\end{aligned}
\end{equation}
and the function $\eta(x_{2},y,z)$ satisfies
\begin{equation}\label{x2}
\begin{aligned}
\eta(x_{2},y,z)=O(x_{2}^{2}+y^{2}+z^{2}) \ \ as\ \ (x_{2},y,z)\rightarrow(0,0,0).
\end{aligned}
\end{equation}
Select a sufficiently small $\epsilon>0$ with $k\geq\sigma_{2}\epsilon$. It follows from (\ref{x1}) and (\ref{x2}) that
\begin{equation}\label{x3}
\begin{aligned}
0<x_{1}(x_{2},y,z)<1,\ \ for\ x_{2}\in[0,\epsilon],\ y\in[0,\epsilon], \ z\in[\sigma_{1}y,\sigma_{2}y].
\end{aligned}
\end{equation}
Now, we define $\gamma\subset\varepsilon_{1}^{u}$ as
$$
\gamma=\left\{(x_{1}(\epsilon,\epsilon,z),\epsilon,\epsilon,z): z\in[\sigma_{1}\epsilon,\sigma_{2}\epsilon]\right\},
$$
and let $p_{1}^{*}=(x_{1}(\epsilon,\epsilon,\sigma_{1}\epsilon),\epsilon,\epsilon,\sigma_{1}\epsilon)$ and $p_{2}^{*}=(x_{1}(\epsilon,\epsilon,\sigma_{2}\epsilon),\epsilon,\epsilon,\sigma_{2}\epsilon)$. Then it is clear that $p_{i}^{*}\in P_{i}$, for $i=1,2$ and $\gamma$ is a curve connecting $p_{1}^{*}$ and $p_{2}^{*}$. Moreover, by (\ref{x3}), it is obvious that $\gamma\backslash\{p_{1}^{*},p_{2}^{*}\}\subset {\rm Int}({\Sigma})$ which is denoted the interior set of $\Sigma$.
\end{proof}

\begin{lemma}\label{L-7}
For each $c \ge c_*$, there is a point $p_{*}\in \varepsilon_{1}^{u}\cap {\rm Int}(\Sigma)$ such that the solution $\Phi_{t}(p_{*})$ of (\ref{profile-2}) through the $p_{*}$ stays in ${\rm Int}(\Sigma)$ for all $t\geq0$.
\end{lemma}
\begin{proof}
 Let $\gamma\subset\varepsilon_{1}^{u}\cap {\rm Int}(\Sigma)$ be defined as in Lemma \ref{L-6}. We define subsets
$\gamma_{1}$ and $\gamma_{2}$ of $\gamma$ as follows:
$$
\gamma_{i}=\{p\in\gamma: \ there\ is \ a \ t\geq0 \ such \ that \ \Phi_{t}(p)\in P_{i}\} \ \ for \ i=1,\ 2.
$$
By Lemma \ref{L-6}, it is clear that $\gamma_{1}$ and $\gamma_{2}$ are nonempty, and by the Lemma \ref{L-5}, the vector field of (\ref{profile-2}) at each point in the plane $P_{i}(i=1,\ 2)$ points to the exterior of $\Sigma$ if $c\ge c_*$. By the continuity of solutions on initial values, we deduce that two sets $\gamma_{1}$ and $\gamma_{2}$ are open relative to $\gamma$. Hence, $\gamma\backslash( \gamma_{1}\cup\gamma_{2})\neq\emptyset$, because of connectedness of $\gamma$. Let $p_{*}\in\gamma\backslash (\gamma_{1}\cup\gamma_{2})$. Then the definitions of $\gamma$ and $\gamma_{i}$ imply that
\begin{align*}
\Phi_{t}(p_{*})\in {\rm Int}(\Sigma)\ \ for\  all\ \ t\geq0.
\end{align*}
\end{proof}
\subsection{Existence and Non-existence of Traveling wave solutions from $E_1$ to $E_*$}
Let us define an essential subsets of $\Sigma$:
$$
G=\{p\in {\rm Int}(\Sigma): \Phi_{p}(t)\in {\rm Int}(\Sigma)\  for\  all\  t\geq0\}.
$$
By Lemma \ref{L-7}, we know that $G$ is nonempty. The desired heteroclitic orbit from $E_{1}$ to $E_{*}$ will be obtained in $G$.
\begin{proposition}\label{P-1}
Let assumption $(H_3)$ hold. Then system (\ref{3pdemodel}) has a traveling wave solution connecting $E_{1}$ and $E_{*}$ for wave speed
$c>c_*$.
\end{proposition}
\begin{proof}
Define the Lyapunov function $L: [0, \infty)\rightarrow \mathbb{R}$, which is associated solutions of (\ref{profile-2}) with initial in $G$,
\begin{equation*}
\begin{aligned}
L(t)=&\frac{a_{21}}{a_{12}}\left(X_{1}-u_{*}\ln X_{1}\right)+\left(X_{2}-v_{*}\ln X_{2}\right)\\
&+\frac{a_{13}a_{21}}{a_{12}a_{31}}\left(Y-w_{*}\ln Y -(Y-Z)(1-\frac{1}{Y})\right),
\end{aligned}
\end{equation*}
and it is clear that $L$ is well defined and continuously differentiable on $G$. Moreover, the orbital derivative of $L(t)$ along (\ref{profile-2}) is
\begin{equation}\label{LDOT}
\begin{aligned}
\dot{L}(t)=&\frac{a_{21}}{a_{12}}\left(\dot{X_{1}}-u_{*}\frac{\dot{X_{1}}}{X_{1}}\right)+\left(\dot{X_{2}}-v_{*}\frac{\dot{X_{2}}}{X_{2}}\right)
+\frac{a_{13}a_{21}}{a_{12}a_{31}}\left(\dot{Z}-w_{*}\frac{\dot{Z}}{Z}-\rho\frac{(Y-Z)^{2}}{Y^{2}}\right)\\
=&\frac{a_{21}}{a_{12}}(X_{1}-u_{*})\left(r_{1}(u_{*}-X_{1})+a_{12}(v_{*}-X_{2})+a_{13}(w_{*}-Y)\right)\\
&+(X_{2}-v_{*})\left(r_{2}(v_{*}-X_{2})+a_{21}(X_{1}-u_{*})\right)\\
& +\frac{a_{13}a_{21}}{a_{12}a_{31}}(Y-w_{*})\left(a_{13}(X_{1}-u_{*})\right)-\frac{\rho a_{13}a_{21}}{a_{12}a_{31}}\frac{(Y-Z)^{2}}{Y^{2}}\\
=&-\frac{a_{21}r_{1}}{a_{12}}(X_{1}-u_{*})^{2}-r_{2}(X_{2}-v_{*})^{2}
-\frac{\rho a_{13}a_{21}}{a_{12}a_{31}}\frac{(Y-Z)^{2}}{Y^{2}}\le0.
\end{aligned}
\end{equation}
Hence, by the LaSalle's Invariance Principle, the $\omega$-limit set of any solution of (\ref{profile-2}) is contained in the largest invariant subset of $\{dL/dt=0\}=\{(u_{*},v_{*},W,W): W>0\}$, which is the singleton $\{(u_{*}, v_{*},w_{*},w_{*})\}$.
This completes the proof.
\end{proof}
\begin{proposition}\label{P-2}
The system (\ref{3pdemodel}) does not have a positive traveling wave solution connecting $E_{1}$ and $E_{*}$ for wave speed $0<c<c_*$.
\end{proposition}
\begin{proof}
If $0<c<c_*$, then $\lambda_{2}$ and $\lambda_{3}$ in (\ref{eigenvalue}) are a pair of complex eigenvalues with positive real parts. By looking at the associated eigenvectors given in \eqref{eigenvectors}, one is able to conclude that if a solution is in the unstable manifold of the equilibrium $E_1$, then its $w(t)$ component can not keep nonnegative all the time when the solution converge to $E_1$ as $t\to-\infty$. This completes the proof.
\end{proof}

\section{Numerical Simulations and Brief Discussions}

We close this work by performing some interesting numerical simulations with \verb-FiPy- which is an object oriented PDE solver, written in Python \cite{FiPy:2009}. Some brief discussions and  biological implications are also given.

\subsection{Numerical Simulations}
Consider system \eqref{3pdemodel} in one-dimensional bounded space $[0, 10]$ with Neumann boundary condition and parameters, $r_1=0.7$, $r_2=0.3$, $\mu=0.15$, $a_{12}=0.15$,  $a_{13}=0.5$, $a_{21}=0.2$, $a_{31}=0.5$. It is easy to check that the assumptions (H1)-(H3) are satisfied for these parameters. Three cases of initial functions are taken in the following.
\begin{enumerate}[(i)]
\item Let $u(x, 0)=0.2$ and $v(x, 0)=0.1$ for all $x$, and
\begin{equation*}
w(x, 0)=\left\{
\begin{aligned}
0.08	 && \text{ if } 4.8\le x\le 5.2,\\
0 && \text{ otherwise.}
\end{aligned}
\right.
\end{equation*}

\item Let $u(x, 0)=1$ for all $x$,
\begin{equation*}
v(x, t)=\left\{
\begin{aligned}
0.1	 && \text{ if } 0\le x\le 5,\\
0 && \text{ otherswise.}
\end{aligned}
\right.
\end{equation*}
\begin{equation*}
w(x, 0)=\left\{
\begin{aligned}
0.1	 && \text{ if } 5\le x\le 10,\\
0 && \text{ otherwise.}
\end{aligned}
\right.
\end{equation*}

\item Let $v(x, 0)=1$ for all $x$,
\begin{equation*}
u(x, t)=\left\{
\begin{aligned}
0.15	 && \text{ if } 0\le x\le 5,\\
0 && \text{ otherwise.}
\end{aligned}
\right.
\end{equation*}
\begin{equation*}
w(x, 0)=\left\{
\begin{aligned}
0.1	 && \text{ if } 0\le x\le 5,\\
0 && \text{ otherwise.}
\end{aligned}
\right.
\end{equation*}
\end{enumerate}

\begin{figure}
\subfigure[initial function]{
   \includegraphics[scale=0.4]{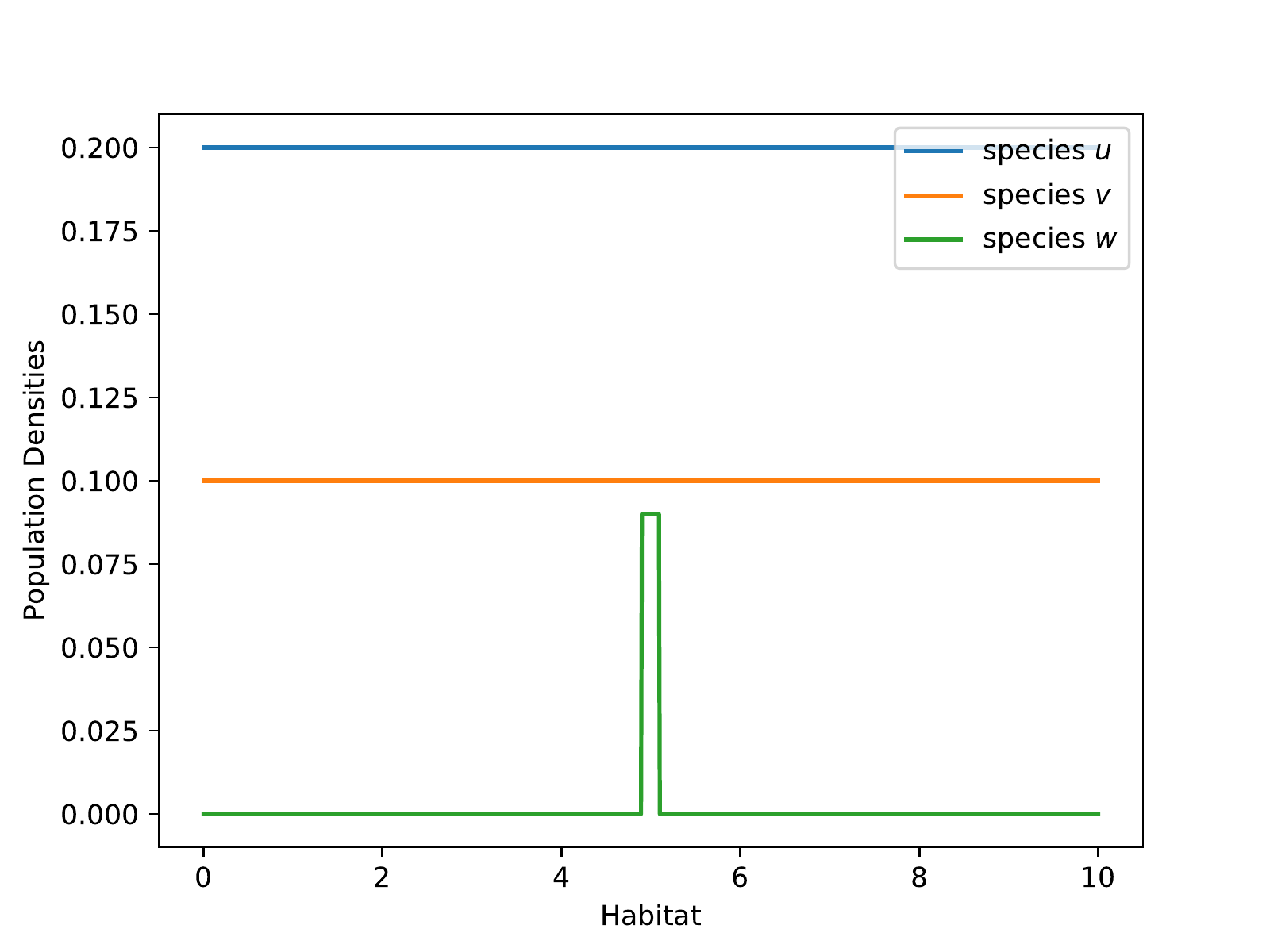}
}
\subfigure[transition state]{
    \includegraphics[scale=0.4]{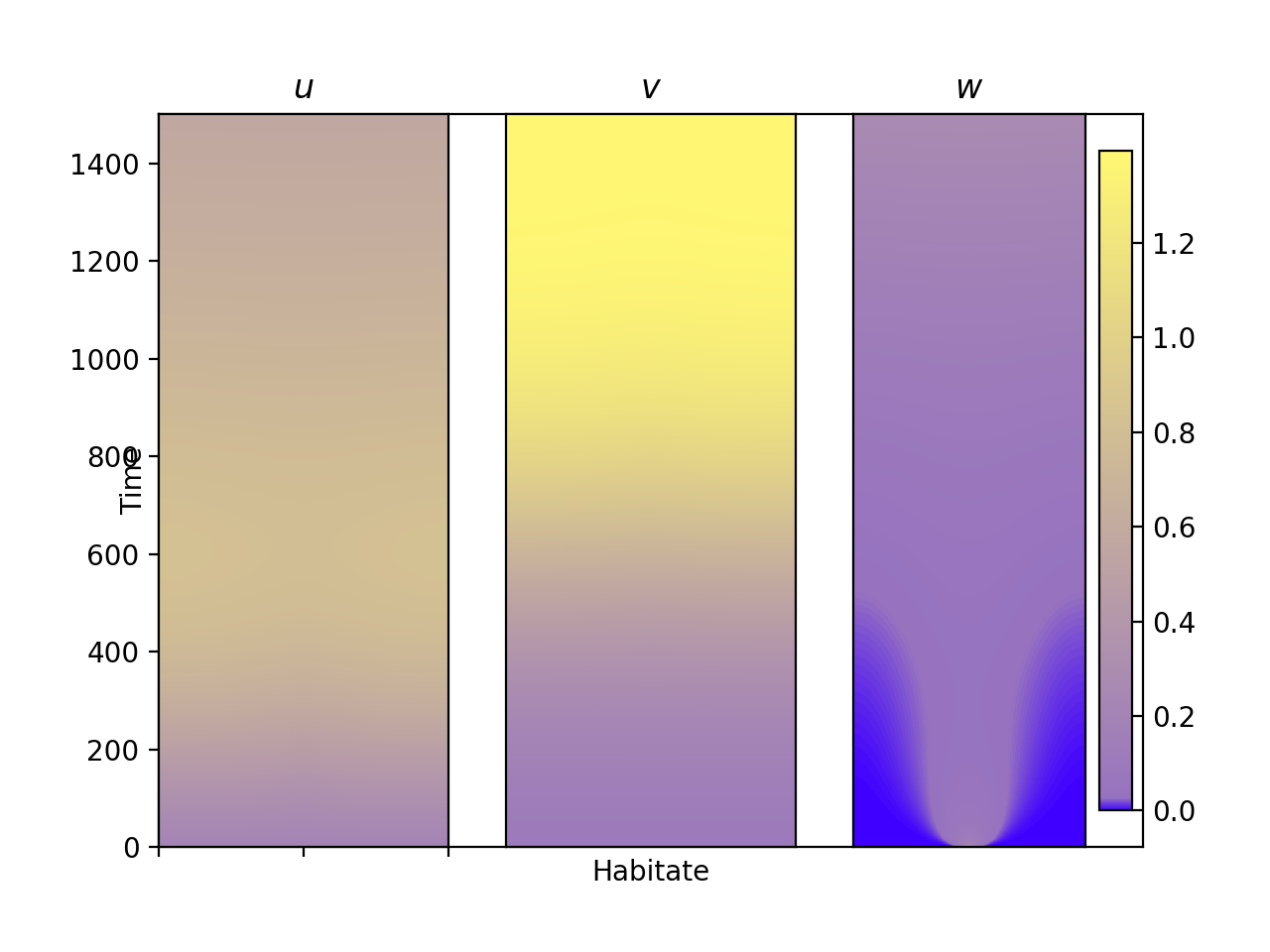}
}
\caption{Panel (a) is the initial functions of species $u$, $v$, and $w$, respectively. In Panel (b), the horizontal axis is the one-dimensional spatial variable represented the habitat, and the vertical axis is the temporal variable. }
\end{figure}

\begin{figure}
\subfigure[initial function]{
   \includegraphics[scale=0.4]{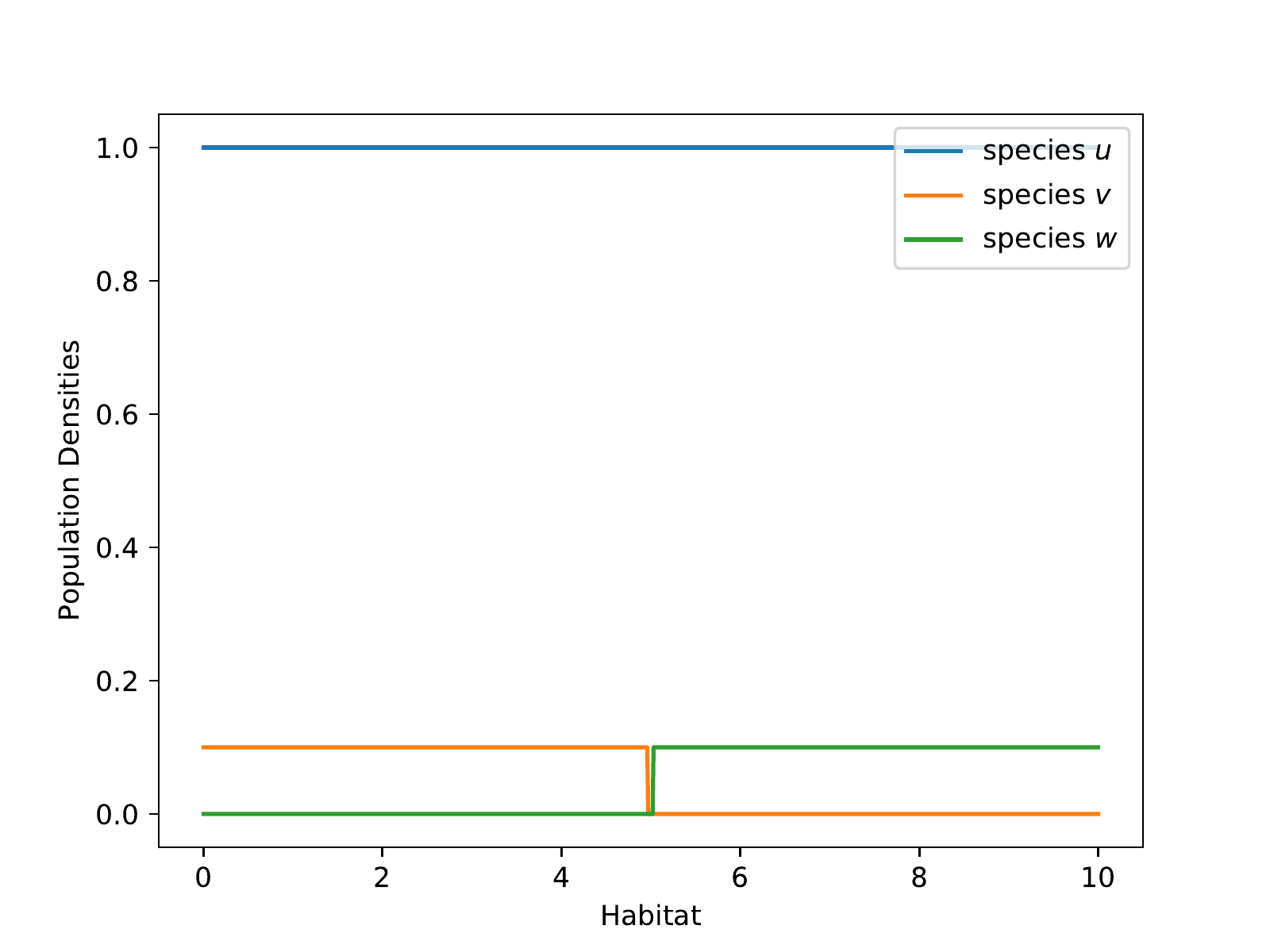}
}
\subfigure[transition state]{
    \includegraphics[scale=0.4]{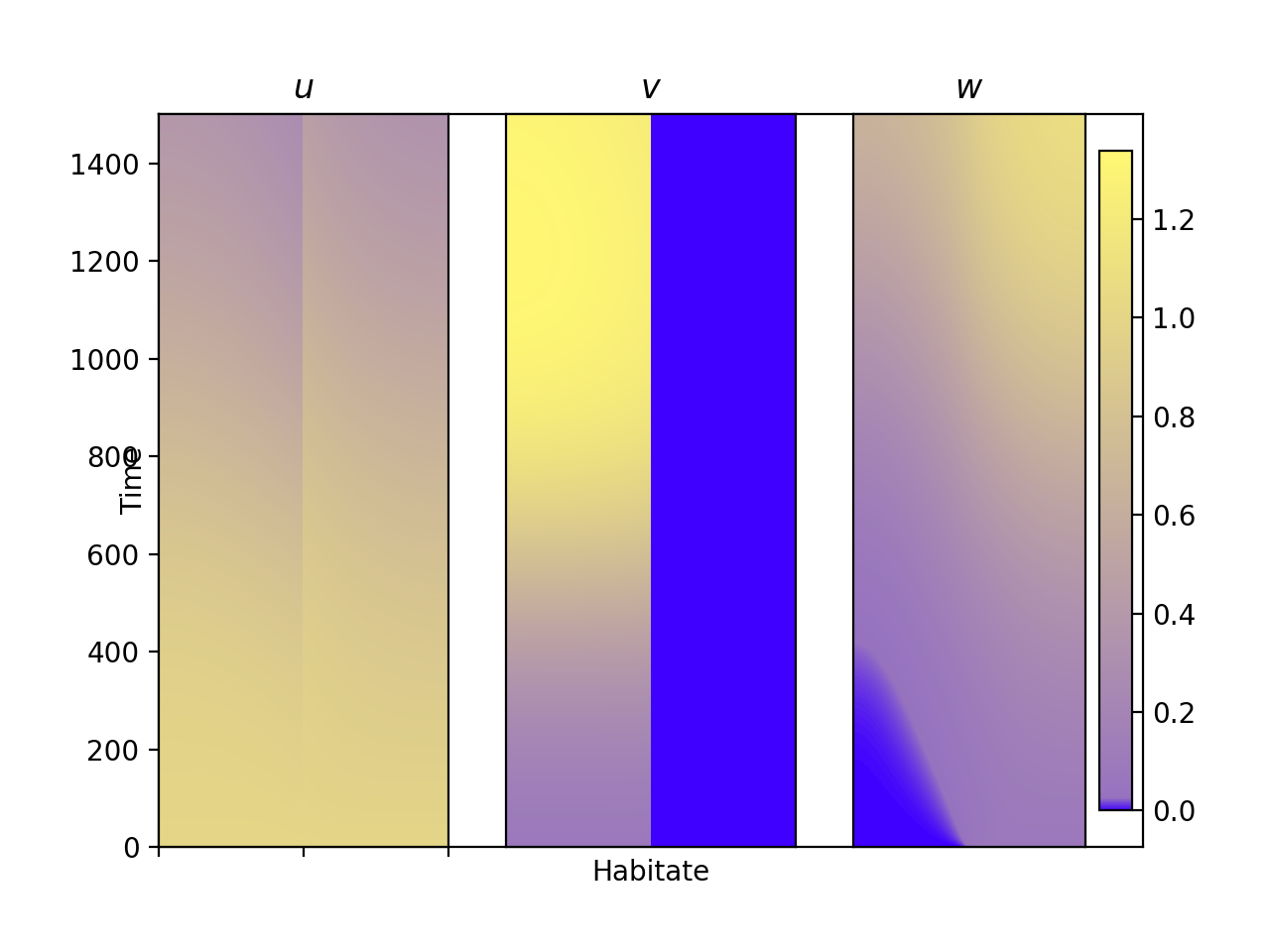}
}
\caption{Panel (a) is the initial functions of species $u$, $v$, and $w$, respectively. In Panel (b), the horizontal axis is the one-dimensional spatial variable represented the habitat, and the vertical axis is the temporal variable.}
\end{figure}

\begin{figure}
\subfigure[initial function]{
   \includegraphics[scale=0.4]{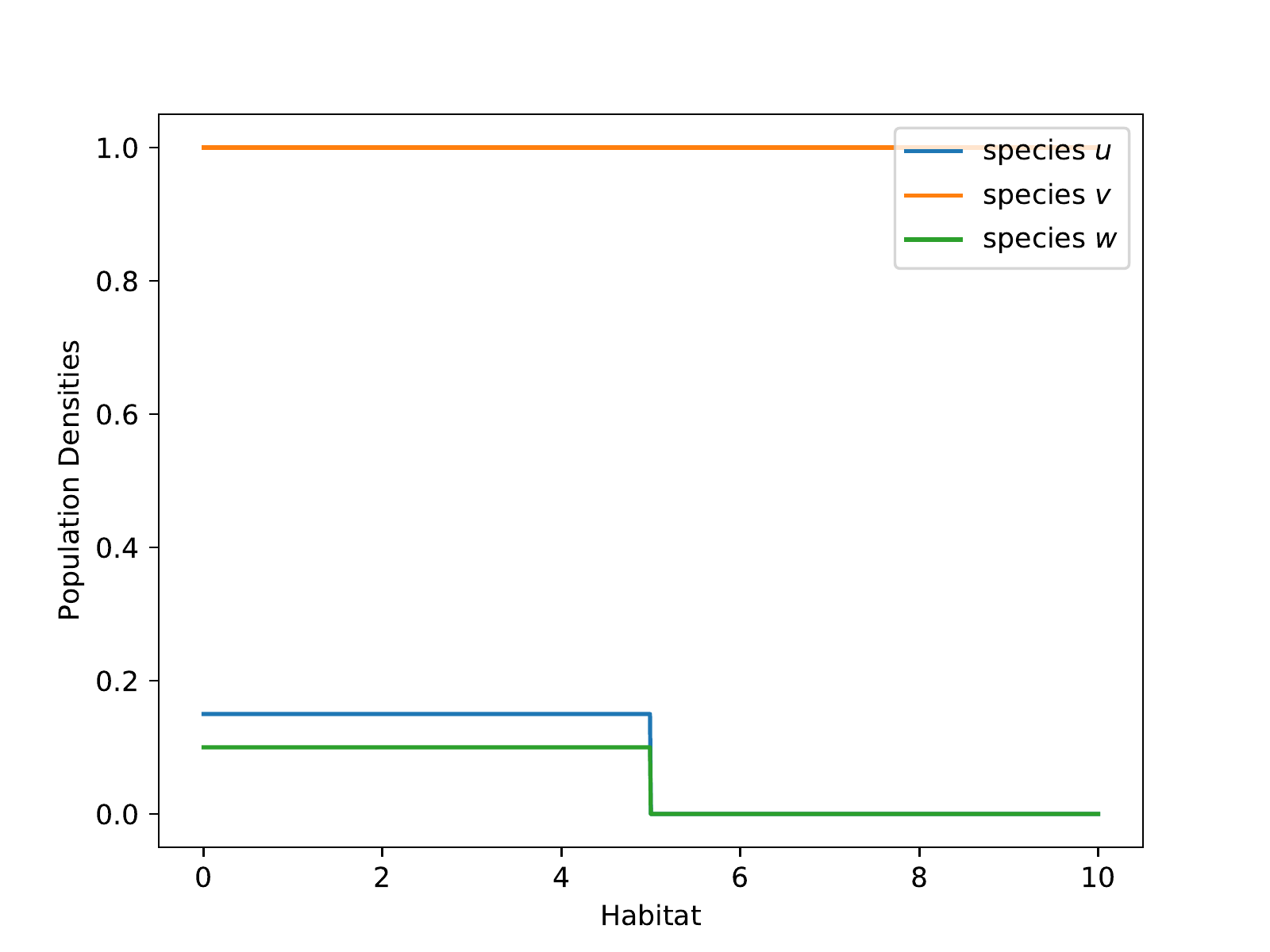}
}
\subfigure[transition state]{
    \includegraphics[scale=0.4]{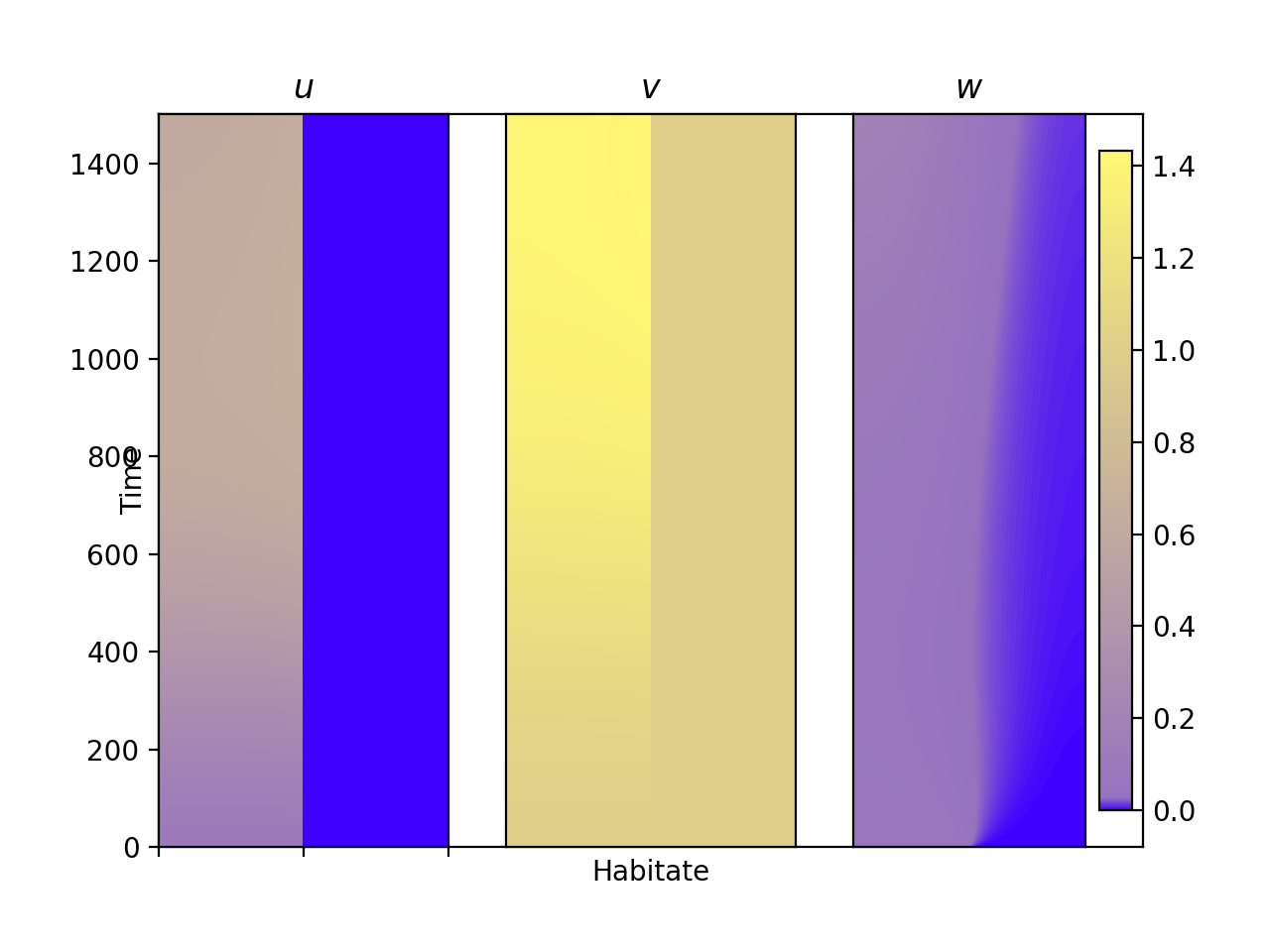}
}
\caption{Panel (a) is the initial functions of species $u$, $v$, and $w$, respectively. In Panel (b), the horizontal axis is the one-dimensional spatial variable represented the habitat, and the vertical axis is the temporal variable.}
\end{figure}
\noindent Please refer Figure 1-3(a) for graphical representations of initial functions.
\medskip

For the first case, we simulate that three species $u$, $v$, and $w$ migrate into a new habitat simultaneously. The initial condition is close to $E_0$ at the  time $t=0$. Note that species $w$ appear in some places of the habitat for $t=0$, and it diffuses to the whole habitat as time increasing. Meanwhile, species $u$ is affected by the diffusion of species $w$ negatively, and species $v$ is affected positively by $u$. So, eventually, all population densities of all species approach to the positive spatial homogeneous state solution $E_*$ of \eqref{3pdemodel}. We can see easily the process of transition from $E_0$ to $E_*$ presented in Figure 1 (b).

For case 2, different initial functions which are close to $E_1$ at time $t=0$ are considered. We would like to simulate that species $v$ and $w$ invade simultaneously into a new habitat with saturated species $u$. Note that species $v$ cannot diffuse, so species $v$ appears in the left half space showed in Figure 2(b) for all time. On the left half space, the transition process from $E_1$ to $E_*$ can be observed clearly. However, on the right half space, we conjecture that the transition process from $E_1$ to $E_{13}$ happens.

For the last case, we set the initial functions close to $E_2$ which is saturated by species $v$. And species $u$ and $w$ migrate to the left half space of the habitat with little populations at time $t=0$. The transition process from $E_2$ to $E_*$ on the left half space can also be observed. On the right half space of habitat, $u=w=0$ for time $t=0$. However, species $w$ can diffuse to the right half space of habitat. We conjecture that the population of $w$ in the right half space of habitat is proportional to the diffusion coefficient $d$.

\subsection{Brief Discussions}

In this work, for system without diffusion \eqref{3odemodel}, we have showed that if (H1) doest not hold then $\lim_{t\to\infty}u(t)=0$ and $\lim_{t\to\infty}w(t)=0$ (Lemma \ref{L-2}). It can be showed that  system \eqref{3odemodel} asymptotically approaches to a one-dimensional subsystem only involving species $v$, and this implies that $E_2$ is globally asymptotically stable by Markus limiting theorem \cite{Markus1956}. If (H2) doest not hold then we only have $\lim_{t\to\infty}w(t)=0$ (Lemma \ref{L-3}). However, we obtain a global results of $E_{12}$ in Theorem \ref{E12GAS} under the assumptions (H1) and $\mu\ge a_{31}u^*_{12}$. It is clearly that $\mu\ge a_{31}$ is a sufficient condition of $\mu\ge a_{31}u^*_{12}$ because of $0<u^*_{12}<1$. Finally, the assumption (H3) is a sufficient and necessary condition for the existence and global stability of the positive equilibrium $E_*$ (Thoerem \ref{EstarGAS}). Since the positive equilibrium $E_*$ does not exist if (H3) is not true, and condition (H3) can be rewritten as the form, 
\[
0< r_{1}r_{2}a_{31}-r_{1}r_{2}\mu-a_{12}a_{31}r_{2}-a_{12}a_{21}\mu=(r_1r_2+a_{12}a_{21})(u^*_{12}a_{31}-\mu).
\]
All global dynamics with respective to essential parameters are summarized in Table 1.
\begin{table}[htp]
\begin{center}
\begin{tabular}{|l|c|c|c|}
\hline
&   $E_2$ & $E_{12}$  & $E_*$\\
\hline
$r_1\le a_{12}$ & GAS & $\nexists$ & $\nexists$ \\
\hline
$r_1> a_{12}$  and $\mu\ge a_{31}u^*_{12}$ & Saddle & GAS & $\nexists$ \\
\hline
$r_1> a_{12}$  and $\mu<a_{31}u^*_{12}$  & Saddle & Saddle & GAS \\
\hline
\end{tabular}
\caption{Classification of dynamics of equilibria with respective to all essential parameters}\label{table1}
\end{center}
\end{table}%

For system with diffusion \eqref{3pdemodel} and using the high-dimensional shooting method, we have show the existence of transition between two different spatial homogeneous states, that is, the existence of traveling wave solution from $E_1$ to $E_*$ for $c>c_*=2\sqrt{d(a_{31}-\mu)}$ which is called the minimal speed. By the similar manner, motivated by Table 1, it is naturally to consider the possibilities of other cases. For example, the traveling wave solution from $E_2$ or $E_{12}$ to $E_*$ under the assumptions $r_1>a_{12}$ and $\mu<a_{31}u^*_{12}$. Or the traveling wave solution from $E_1$ or $E_2$ to $E_{12}$ under the assumptions $r_1>a_{12}$ and $\mu\ge a_{31}u^*_{12}$. We left these open problems to be considered.
%\begin{enumerate}
%\item Classification all possible transition states form one unstable equilibrium to another stable equilibrium.
%\begin{table}[htp]
%\caption{Possible}
%\begin{center}
%\begin{tabular}{|c|c|c|c|c|c|c|}
%\hline
%& $E_0$ & $E_1$ & $E_2$ & $E_{12}$ & $E_{13}$ & $E_*$\\
%\hline
%$E_0$ & $\times$ &$\times$&$\times$&  & $\times$ & \\
%\hline
%$E_1$ &$\times$& $\times$ &$\times$&& \checkmark& \checkmark \\
%\hline
%$E_2$ &$\times$&& $\times$ & \checkmark&$\times$&\\
%\hline
%$E_{12}$ & $\times$&&& $\times$ &$\times$&\\
%\hline
%$E_{13}$ &$\times$&&&& $\times$ &\\
%%\hline 
%%$E_*$ &&&&&& $\times$\\
%\hline
%\end{tabular}
%\end{center}
%\label{default}
%\end{table}%
%\item Is $c_*$ the minimal speed?
%\end{enumerate}

%%%%%%%%%%%%%%%%%%%%%%%%%%%%%%%%%%
%\bibliographystyle{abbrv}
%\bibliography{TWS-1214}

\end{document}